\documentclass[reqno]{amsart}

\usepackage[english]{babel}
\usepackage{amsfonts,amsmath,amssymb,dsfont,csquotes}
\usepackage{pdfsync}
\usepackage{graphicx}
\usepackage{mathtools}
\usepackage{enumerate}
\usepackage{cases}

\usepackage{xcolor}

\definecolor{gr}{rgb}   {0.,   0.69,   0.23 }
\definecolor{bl}{rgb}   {0.,   0.5,   1. }
\definecolor{mg}{rgb}   {0.85,  0.,    0.85}
\definecolor{or}{rgb}   {0.9,  0.5,   0.}

\definecolor{webred}{rgb}{0.75,0,0}
\definecolor{webgreen}{rgb}{0,0.75,0}
\usepackage[citecolor=webgreen,colorlinks=true,linkcolor=webred]{hyperref}

\newtheorem{theorem}{Theorem}[section]
\newtheorem{proposition}[theorem]{Proposition}
\newtheorem{lemma}[theorem]{Lemma}
\newtheorem{corollary}[theorem]{Corollary}

\theoremstyle{definition}
\newtheorem{definition}[theorem]{Definition}
\newtheorem{notation}[theorem]{Notation}

\theoremstyle{remark}
\newtheorem{remark}[theorem]{Remark}

\newcommand{\Bk}{\color{black}}

\newcommand{\beq}{\begin{equation}}
\newcommand{\eeq}{\end{equation}}
\newcommand{\ba}{\begin{array}}
\newcommand{\ea}{\end{array}}
\newcommand{\bea}{\begin{eqnarray}}
\newcommand{\eea}{\end{eqnarray}}
\newcommand{\beas}{\begin{eqnarray*}}
\newcommand{\eeas}{\end{eqnarray*}}

{

\makeatletter
\def\paragraph{\@startsection{paragraph}{4}%
  \z@{0.3em}{-.5em}%
  {$\bullet$ \ \normalfont\itshape}}
\makeatother

\newcommand{\gb}{\mathfrak{b}}
\newcommand{\N}{\mathbb{N}}
\newcommand{\R}{\mathbb{R}}

\newcommand{\bB}{{\bf B}}
\newcommand{\bA}{{\bf A}}
\newcommand{\gh}{\mathfrak{h}}
\newcommand{\gq}{\mathfrak{q}}
\newcommand{\gS}{\mathsf{sp}}
\newcommand{\dx}{\,\mathsf{d}}
\newcommand{\one}{\mathds{1}}
\newcommand{\ess}{\mathsf{ess}}
\newcommand{\sL}{\mathsf{L}}
\newcommand{\sH}{\mathsf{H}}
\newcommand{\Dom}{\mathsf{Dom}}
\newcommand{\loc}{\mathsf{loc}}
\newcommand{\spann}{\mathsf{span}\,}
\newcommand{\lef}{\mathsf{le}}
\newcommand{\eps}{\varepsilon}
\newcommand{\rig}{\mathsf{ri}}

\newcommand{\rank}{\mathsf{rank\,}}

\begin{document}

\title[Band functions in the presence of magnetic steps]{Band functions in the presence of magnetic steps}

\author[P.\ D.\ Hislop]{P.\ D.\ Hislop}
\address{Department of Mathematics,
    University of Kentucky,
    Lexington, Kentucky  40506-0027, USA}
\email{peter.hislop@uky.edu}

\author[N.\ Popoff]{N.\ Popoff}
\address{D\'epartement de Math\'ematiques, Universit\'e de Bordeaux}
\email{nicolas.popoff@cpt.univ-mrs.fr}

\author[N.\ Raymond]{N.\ Raymond}
\address{Institut de Recherche Math\'ematique de Rennes. Universit\'e de Rennes 1, France}
\email{nicolas.raymond@univ-rennes1.fr}

\author[M.\ P. \ Sundqvist]{M.\ P. Sundqvist}
\address{Centre for Mathematical Sciences, Box 118, SE-22100, Lund, Sweden}
\email{mickep@maths.lth.se}

\thanks{PDH was partially supported by NSF through grant DMS-1103104.}

\thanks{Version of \today}

\begin{abstract}
  We complete the analysis of the band functions for two-dimensional magnetic Schr\"o\-dinger operators with piecewise constant magnetic fields. The discontinuity of the magnetic field can create edge currents that flow along the discontinuity that have been described by physicists. Properties of these edge currents are directly related to the behavior of the band functions. The effective potential of the fiber operator is an asymmetric double well (eventually degenerated) and the analysis of the splitting of the bands incorporates the asymmetry. If the magnetic field vanishes, the reduced operator has essential spectrum and we provide an explicit description of the band functions located below the essential spectrum. For non degenerate magnetic steps, we provide an asymptotic expansion of the band functions  at infinity. We prove that when the ratio of the two magnetic fields is rational, a splitting of the band functions occurs and has a natural order, predicted by numerical computations.
 \end{abstract}

\maketitle \thispagestyle{empty}



\vspace{.2in}


\section{Introduction and motivation}

In this paper, we investigate the spectral properties of magnetic Hamiltonians
\beq\label{eq:magnetic-ham1}
\mathcal{L}_{\gb}=(-i\nabla-\bA)^2=D_{x}^2+(D_{y}-a_{\gb}(x))^2 ,
\eeq
in $\sL^2 (\R^2)$ with the following discontinuous magnetic field
\[
\bB(x,y)=b_{1}\one_{\R_{-}}(x)+b_{2}\one_{\R_{+}}(x),
\]
where $\gb=(b_{1},b_{2})\in\R^2$, and $\one_X(x)$ is the characteristic function on $X \subset \R$. An associated vector potential is given by :
\[
\bA(x,y)=(0,a_{\gb}(x)),\qquad a_{\gb}(x)=b_{1}x\one_{\R_{-}}(x)+b_{2}x\one_{\R_{+}}(x).
\]
These operators naturally appear in several models involved in nanophysics such as the Ginzburg--Landau model of superconductivity (see \cite{SJST}) and in quantum transport in a bidimensional electron gas. In particular in an unbounded system, inhomogenous magnetic fields may play the role of a quantum waveguide (\cite{PeetMat93,PeetRej}). More recently it appears that such magnetic barriers also induce transport and act as waveguides in graphene by creating snake states along magnetic discontinuities (see \cite{Or08} and \cite{Gho08} where piecewise constant magnetic fields are considered). These physical properties can be described with the corresponding magnetic Schr\"odinger operators that exhibit a variety of interesting spectral and transport properties. In particular, there may exist edge currents that propagate along the edge at~\mbox{$x=0$}. The conductivity corresponding to these edge currents is, in fact, quantized. Most of these properties can be deduced from the study of the band functions associated with the Hamiltonian (see for instance \cite{ManPur97,DomGerRai}). We will describe later some consequences of our analysis of band functions.

The symmetries of the Hamiltonian allow us to classify various configurations of the magnetic field $(b_1,b_2)$. We notice that $\overline{\mathcal{L}_{\gb}}=\mathcal{L}_{-\gb}$ so that we may assume that $b_{2}\geq 0$. If $S$ denotes the symmetry $(x,y)\mapsto (-x,-y)$, we have $S\mathcal{L}_{b_{1},b_{2}}S=\mathcal{L}_{-b_{2},-b_{1}}=\overline{\mathcal{L}_{b_{2},b_{1}}}$. For $B>0$, we introduce the $\sL^2$-unitary transform
\[
U_{B}\psi(x,y)=B^{-1/2}\psi(B^{-1/2}x,B^{-1/2}y)
\]
and we have
\[
U_{B}^{-1}\mathcal{L}_{\gb}U_{B}=B^{-1}\mathcal{L}_{B \gb}.
\]
These considerations allow the following reductions:
\begin{enumerate}
\item If $b_{1}$ or $b_{2}$ is zero we may assume that $b_{1}=0$ and, if $b_{2}\neq 0$, we may assume that $b_{2}=1$. We call the case $(b_{1},b_{2})=(0,1)$ the \enquote{magnetic wall}.
\item If $b_{1}$ and $b_{2}$ have opposite signs and $|b_{1}|\neq |b_{2}|$, it
is sufficient to consider the case when $-1<b_{1}<0<b_{2}=1$. We call this case the \enquote{trapping magnetic step} (the term trapping will appear clearer after stating our results).
\item If $|b_{1}|= |b_{2}|$, we may restrict ourselves to the cases $(b_{1},b_{2})=(1,1)$ and $(b_{1},b_{2})=(-1,1)$. This last one may be called the \enquote{symmetric trapping magnetic step} or \enquote{magnetic barrier}.
\item If $b_{1}$ and $b_{2}$ have the same sign, we may assume that $0<b_{1}<b_{2}=1$. We call this case the \enquote{non-trapping magnetic step}.
\end{enumerate}

A physical description of the distinct classical orbits associated with these magnetic field configurations is provided in~\cite{PeetVas93,PeetRej}. The last two cases, (3) and (4), have already been studied in the mathematical literature (see the overview below). Our goal is to provide a mathematical treatment of the unstudied cases (1) and (2). We will also give the consequences of our results on the behavior of quantum particles submitted to the magnetic field ${\gb}$.

\subsection{The band functions and their link with the magnetic Laplacian}\label{sec:fiber-decomp1}
\paragraph{Direct fiber decomposition}
In order to perform the spectral analysis, we can use the translation invariance in the $y$-direction and thus the direct integral decomposition (see \cite[XIII.16]{ReSi78}) associated with the Fourier transform with respect to $y$, denoted by $\mathcal{F}_{y}$,
\[
\mathcal{F}_{y}\mathcal{L}_{\gb}\mathcal{F}_{y}^{-1}=\int_{k\in\R}^{\oplus} \gh_{\gb}(k) \dx k,
\]
where
\[
\gh_{\gb}(k)=D_{x}^2+V_{\gb}(x,k),\qquad \mbox{ with }\qquad V_{\gb}(x,k)=(k-a_{\gb}(x))^2.
\]
Note that the effective potential $V_{\gb}(x,k)$ has the form
\beq\label{eq:eff-potential1}
V_{\gb}(x,k) =
\begin{cases}
(k-b_1x)^2, & x < 0, \\
(k-b_2x)^2, & x>0.
\end{cases}
\eeq

\begin{figure}[htbp]
\centering
\makebox[\textwidth]{%
\includegraphics[width=6cm]{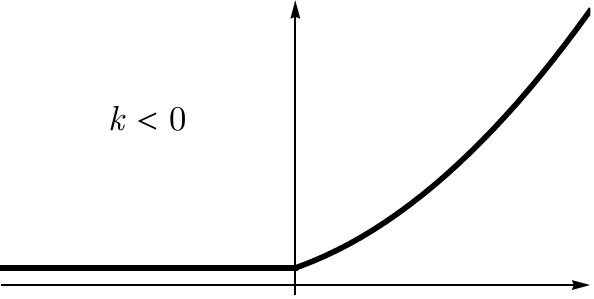}
\hskip 0.5cm
\includegraphics[width=6cm]{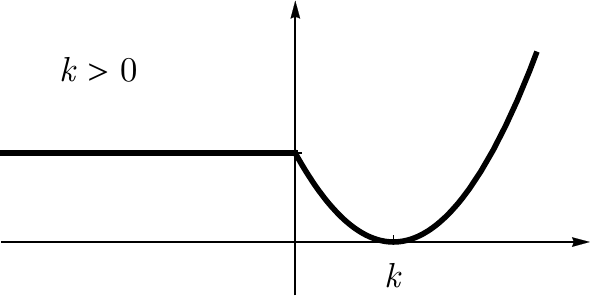}
}
\caption{The potential $V_{\gb}(x,k)$ corresponding to case $(1)$.}
\label{fig:case1}
\end{figure}

\begin{figure}[htbp]
\centering
\makebox[\textwidth]{%
\includegraphics[width=6cm]{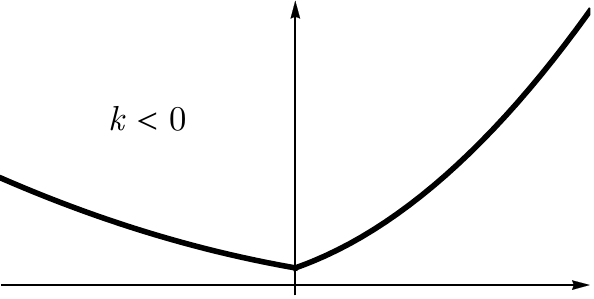}
\hskip 0.5cm
\includegraphics[width=6cm]{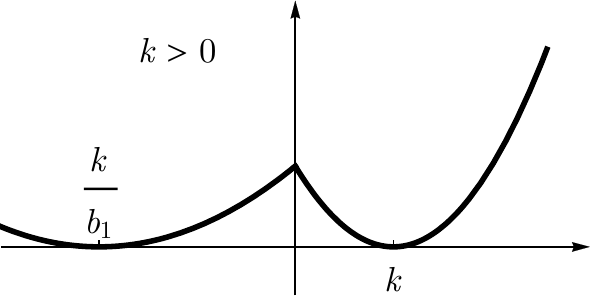}
}
\caption{The potential $V_{\gb}(x,k)$ corresponding to case $(2)$.}
\label{fig:case2}
\end{figure}

\begin{figure}[htbp]
\centering
\makebox[\textwidth]{%
\includegraphics[width=6cm]{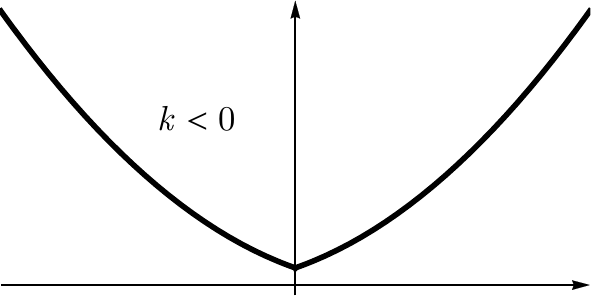}
\hskip 0.5cm
\includegraphics[width=6cm]{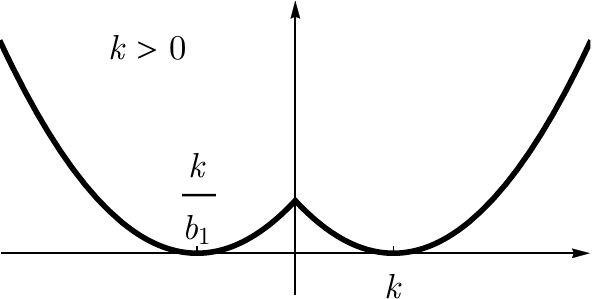}
}
\caption{The potential $V_{\gb}(x,k)$ corresponding to case $(3)$ with $b_1=-1$ and $b_2=1$.}
\label{fig:case3a}
\end{figure}


\begin{figure}[htbp]
\centering
\makebox[\textwidth]{%
\includegraphics[width=6cm]{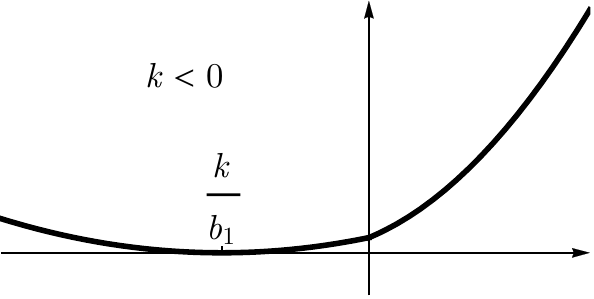}
\hskip 0.5cm
\includegraphics[width=6cm]{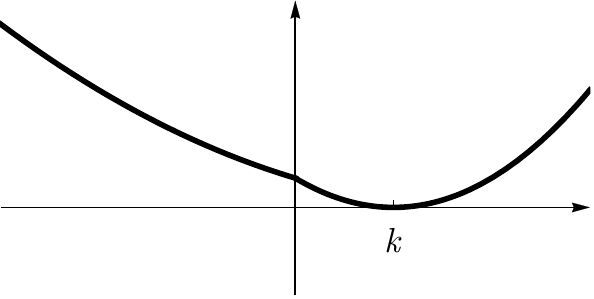}
}
\caption{The potential $V_{\gb}(x,k)$ corresponding to case $(4)$.}
\label{fig:case4}
\end{figure}

The domain of $\gh_{\gb}(k)$ is given by
\[
\Dom(\gh_{\gb}(k))=\{\psi\in\Dom(\gq_{\gb}(k)) : (D_{x}^2+V_{\gb}(x,k))\psi\in\sL^2(\R)\},
\]
where the quadratic form $\gq_{\gb}(k)$ is defined by
\[
\gq_{\gb}(k)(\psi)=\int_{x\in\R} |\psi'(x)|^2+V_{\gb}(x,k)|\psi|^2\dx x.
\]
The spectrum of $\mathcal{L}_{\gb}$, $\gS(\mathcal{L}_{\gb})$, satisfies
\[
\gS(\mathcal{L}_{\gb})=\overline{\bigcup_{k\in\R}\gS(\gh_{\gb}(k))}.
\]
We now define the $n$-th band function below the essential spectrum:
\begin{notation}
We denote by $\lambda_{\gb,n}(k)$ the $n$-th Rayleigh quotient of $\gh_{\gb}(k)$. We recall that if $\lambda_{\gb,n}(k)$ is strictly less than the infimum of the essential spectrum, it coincides with the $n$-th eigenvalue of $\gh_{\gb}(k)$.
\end{notation}

\paragraph{State of the art}

The cases (3) and (4) have already been extensively studied. Case~(3) for which $(b_1,b_2) = (1,1)$ is just the Landau Hamiltonian. Its spectrum is pure point and consists of the so-called Landau levels; infinitely degenerated eigenvalues $(2n-1)$, for $n \geq 1$. The band functions are constants $\lambda_{\gb,n}(k) = 2n-1$.

The case (3), for which $(b_1,b_2) = (-1,1)$, can be linked to operators acting on a half-line. We first need to introduce the standard de~Gennes operators, that will also be used in some results of this article. We recall below well-known properties of the spectrum of these operators.

\begin{notation}
We let $E_{0}=0$ and for $n\geq 1$, $E_{n}=2n-1$.
\end{notation}

\begin{notation}
\label{ntn:theta1}
For $\circ={\rm N}$, (resp. $\circ={\rm D}$), let us introduce $\gh^{\circ}(k)$ the Neumann (resp. Dirichlet) realization on $\R_{+}$ of $D_{t}^2+(t-k)^2$ and its eigenvalues $(\mu^\circ_{\ell}(k))_{\ell\geq 1}$. For all $\ell\geq 1$, the function $\mu_{\ell}^{\rm N}$ admits a unique and non-degenerate minimum at $k=\xi_{\ell-1}$, denoted by $\Theta_{\ell-1}$, with $\Theta_{\ell-1} \in (E_{\ell-1},E_{\ell})$. Moreover, $\xi_{\ell-1}$ is also the unique solution of the equation $\mu^{\rm N}_{\ell}(k)=k^2$.
The function $\mu_{\ell}^{\rm D}(k)$ is a decreasing function of $k$.
Both $\mu_{\ell}^{\rm N}(k)$ and $\mu_{\ell}^{\rm D}(k)$ go to $+\infty$ as $k\to-\infty$ and go to $E_{\ell}$ as $k\to+\infty$.
\end{notation}
We refer mainly to \cite{dBiePu99} for the Dirichlet operator and \cite{DauHe93,HeMo01} for the Neumann one. All these properties can also be found in \cite{FouHel10,Ivrii}.

Using symmetry arguments, it is proved in~\cite{DomHisSoc13,Popoff} that for $\gb=(-1,1)$ and $\ell\geq1$,
\[
\lambda_{\gb,2\ell-1}(k) = \mu_{\ell}^{\rm N} (k) \quad \text{and} \quad \lambda_{\gb,2\ell} (k) = \mu_{\ell}^{\rm D} (k)\quad k \in \R,
\]
linking the de Gennes operator with the model with symmetric trapping magnetic steps.
This analysis of the band functions provides a description of a quantum particle localized in energy far from the thresholds $\{ E_n \}$
and submitted to the system: they can exhibit edge current that may be viewed as quantum manifestations of the classical
snake orbits \cite{PeetRej} along the edge $x=0$. These edge currents are exponentially decreasing far from the edge. The corresponding edge currents and the asymptotics of the band functions are described in \cite{DomHisSoc13}. The effective potential (see \eqref{eq:eff-potential1}) has the form of two symmetric double wells for $k > 0$. The fact that each band function is exponentially close to a Landau level as $k \rightarrow +\infty$ is called a {\it splitting} of the band function. The analysis of quantum states localized in energy near the thresholds is made in \cite{HisPofSoc14} for one edge quantum hall system and relies on a precise asymptotic analysis of the band function $\mu_{\ell}^{\rm D}(k)$ as $k \to +\infty$. The band functions associated to this case are displayed in Figure~\ref{F:1}.

The case (4), a non-trapping magnetic step, was treated in the initial paper by Iwatsuka \cite{Iwa85} (along with the more general situation for which the magnetic field depends only on one direction and simply has constant values at $\pm
 \infty$). In this case, all the band functions are monotone nondecreasing. The analysis of the associated edge current is done in \cite{HisSoc15}.

Notice that all cases (1)--(4) are described from a physical point of view in \cite{PeetRej}.

\begin{figure}
\centering
\includegraphics[width=10cm]{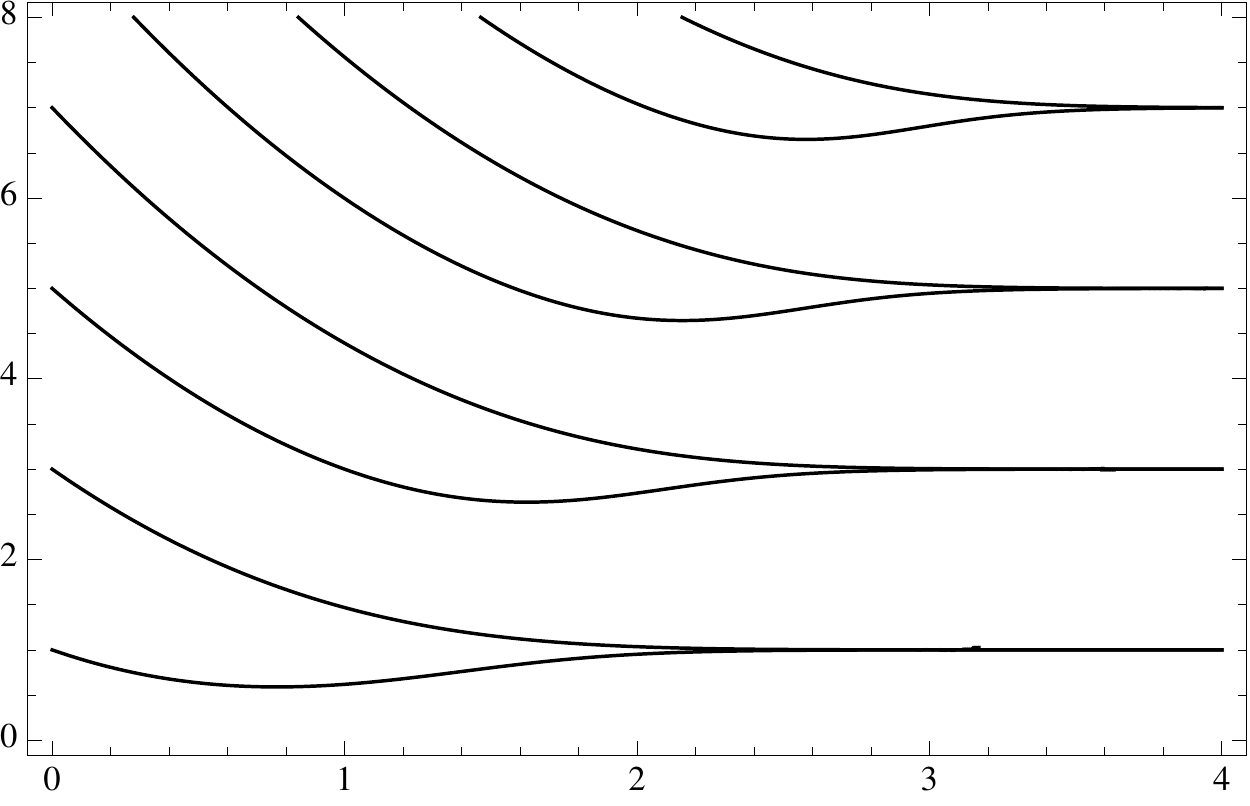}
\caption{Band functions for the symmetric double-well case $\gb=(-1,1)$.}
\label{F:1}
\end{figure}

\subsection{Main results}\label{subsec:main1}

We first consider case (1) for which $(b_1, b_2) = (0, 1)$.
This was not treated in \cite{HisSoc13} or in \cite{DomHisSoc13}. It is physically interesting because the classical orbits of an electron originating in the half-plane $x < 0$ are unbounded straight half-lines or lines (depending upon whether the orbit intersects $x=0$ or not), whereas those in the half-plane $x>0$ are bounded circles provided the Lamor orbit does not cross the edge at $x=0$.

\begin{theorem}\label{magnetic-wall}
Let $\gb = (0,1)$.

If $k<0$, then the fiber operator $\gh_{\gb}(k)$ has only purely absolutely continuous spectrum $[k^2, \infty)$. Consequently, the spectrum of $\mathcal{L}_{\gb}$ is $[0, \infty)$.

If $k > 0$, the essential spectrum of $\gh_{\gb}(k)$ is $[k^2, \infty)$. The band functions $\lambda_{\gb,n}$ are nondecreasing functions with $E_{n-1} < \lambda_{\gb,n}(k) < E_n$ and $\lim_{k \rightarrow \infty} \lambda_{\gb,n}(k) = E_n= 2n-1$. For $n\geq 1$ and $\xi_{n-1}<k<\xi_{n}$ (recall Notation \ref{ntn:theta1}), the operator $\gh_{\gb}(k)$ admits exactly $n$ simple eigenvalues below the threshold of its essential spectrum.
\end{theorem}

The case (2) of a trapping magnetic step is technically more challenging. It relies on the asymptotics of the eigenfunctions of a Schr\"odinger operator with an asymmetric double-well potential. The results concerning the splitting of two energy levels depends upon the ratio of the two constant magnetic fields $b_1 / b_2 = b_1$ since $b_2=1$ by normalization.

\begin{definition}\label{def}
We define $\mathfrak{L}=\left\{E_{n},n\geq 1\}\bigcup\{|b_{1}|E_{n}, n\geq 1\right\}$ and we define the splitting set $\mathfrak{S}=\left\{E_{n},n\geq 1\}\bigcap\{|b_{1}|E_{n}, n\geq 1\right\}$. We will say that there is a \emph{splitting of levels} if $\mathfrak{S}\neq\emptyset$ i.e. if there exist positive integers $m<n$ such that
\beq\label{eq:split1}
b=\frac{2n-1}{2m-1},\quad \mbox{ with } b=|b_{1}|.
\eeq
\end{definition}
As we shall see, the problem of the behavior for large $k$ of the band functions can be reformulated into a semiclassical problem for a one dimensional Schr\"odinger operator with multiple wells (see for instance \cite{HeSj84,HeSjIII}). In particular the limits of the band functions are precisely the elements of $\mathfrak{L}$, counted with multiplicty. Roughly speaking, either one or two band functions converge toward each element of $\mathfrak{L}$.

However our investigation will not rely on the semiclassical tools: We will use asymptotic expansions of special functions in order to derive the asymptotic behavior of the band functions. Let us now state our main theorem.

\begin{theorem}\label{magnetic-step}
Let $\gb=(b_1, 1)$, with $-1<b_1<0$.
For $k \in \R$, the fiber operator $\gh_{\gb}(k)$ has only purely discrete spectrum $\lambda_{\gb,n}(k)$ with simple eigenvalues, and $\lim_{k \to - \infty} \lambda_{\gb,n}(k) = + \infty$. The set of limit points of the band functions $\lambda_{\gb,n}(k)$ as $k \to +\infty$ is precisely $\mathfrak{L}$.
Moreover, we may describe the asymptotics of the band functions in the following way.
\begin{enumerate}[(i)]
\item\label{I} Non-splitting case.  Let $\lambda\in \mathfrak{L}\setminus\mathfrak{S}$. We can write in a unique way $\lambda=E_{n}$ or $\lambda=bE_{n}$ with $n\geq1$. 
Then there exists a unique $p_{n}\in \N^{*}$ such that:
\begin{enumerate}[(a)]
\item\label{1} $\lambda_{\gb,p_{n}}(k)-E_{n}=\eps_{n}(k)\underset{k\to+\infty}{\to} 0$ (in the case $\lambda=E_{n}$),
\item\label{2} $\lambda_{\gb,p_{n}}(k)-bE_{n}=\eps_{n}(k)\underset{k\to+\infty}{\to} 0$ (in the case $\lambda=bE_{n}$).
\end{enumerate}
In the first case,
\[
\eps_{n}(k)=-\frac{1}{\sqrt{\pi}}2^{n-2}(1+b)\frac{1}{(n-1)!}k^{2n-3}e^{-k^2}(1+o(1)),
\]
and in the second case,
\[
 \eps_{n}(k)=-\frac{1}{\sqrt{\pi}}2^{n-2}b^{3/2-n}(1+b)\frac{1}{(n-1)!}k^{2n-3}e^{-k^2/b}(1+o(1)).
\]

\item\label{II} Splitting case.  Assume that the splitting set is not empty and consider $\lambda\in\mathfrak{S}$. Let us consider $n\geq 1$ and $m\geq 1$ such that $\lambda=E_{n}=bE_{m}$. There exists $p_{n}\in \N^{*}$ such that
\[
\left\{
\begin{aligned}
&\lambda_{\gb,p_{n}}(k)-E_{n}=\eps_{n}^{-}(k)
\\
&\lambda_{\gb,p_{n}+1}(k)-E_{n}=\eps_{n}^{+}(k)
\end{aligned}
\right.
\]
with, as $k$ tends to $+\infty$,
\begin{align*}
\eps_{n}^{-}(k)&=-\frac{2^{n-3/2}(1+b)}{(n-1)!\sqrt{\pi}}k^{2n-3}e^{-k^2}(1+o(1)),
\\
\eps_{n}^{+}(k)&=\frac{2^{m+\frac{3}{2}}b^{-m+\frac{3}{2}}}{(m-1)!(1+b)\sqrt{\pi}}k^{2m+1}e^{-k^2/b}(1+o(1)).
\end{align*}
\end{enumerate}
\end{theorem}
The energy levels for which the exponentially small splitting
occurs when $k \rightarrow \infty$ behave like the band functions in the
symmetric magnetic field case case $b_1 = -1$ and $b_2 = 1$. In that case (see  \cite{DomHisSoc13}), two band functions approach the Landau level $2n-1$ as $k \rightarrow \infty$, one is monotone decreasing and the other has a unique, nondegenerate minimum (see Figure \ref{F:1}). In the present article, with the help of our asymptotic expansions, we may deduce the existence of a minimum for some band functions.
\begin{corollary}
\label{C:non-mono}
With the same assumptions and notations as in Theorem \ref{magnetic-step}, the following statements hold:
\begin{enumerate}[(i)]
\item Let $\lambda\in \mathfrak{L}\setminus\mathfrak{S}$. Then each band function converging to $\lambda$ admits a global minimum.
\item Let $\lambda\in \mathfrak{S}$. Then $k\mapsto \lambda_{p_{n}}(k)$ admits a global minimum.
\end{enumerate}
\end{corollary}
Notice that the first band function always admits a global minimum. Moreover if $|b_{1}|$ is not a ratio of odd integers, all the band functions admit a global minimum. We conjecture that the only case when monotonicity occurs is Case $(ii)$ of Theorem \ref{magnetic-step} and corresponds to $\lambda_{p_{n}+1}$. In that case, the non-monotonicity of each band function depends on a non trivial way on the index of the eigenvalue. This situation seems striking to us, especially compared to other known case (see \cite{DomHisSoc13,HisSoc15,Yaf08}).

\subsection{Numerical computations}

We display on Figures \ref{F:2}--\ref{F:4} the band functions associated with the magnetic step $\gb=(-\frac{1}{3},1)$ together with the asymptotics provided below as $k\to+\infty$. The splitting occurs near each value in $\mathfrak{S}=\{2n-1, n\in \N^{*}\}$. These graphs have been obtained by using Mathematica.

\begin{figure}[htbp]
\centering
\includegraphics[width=10cm]{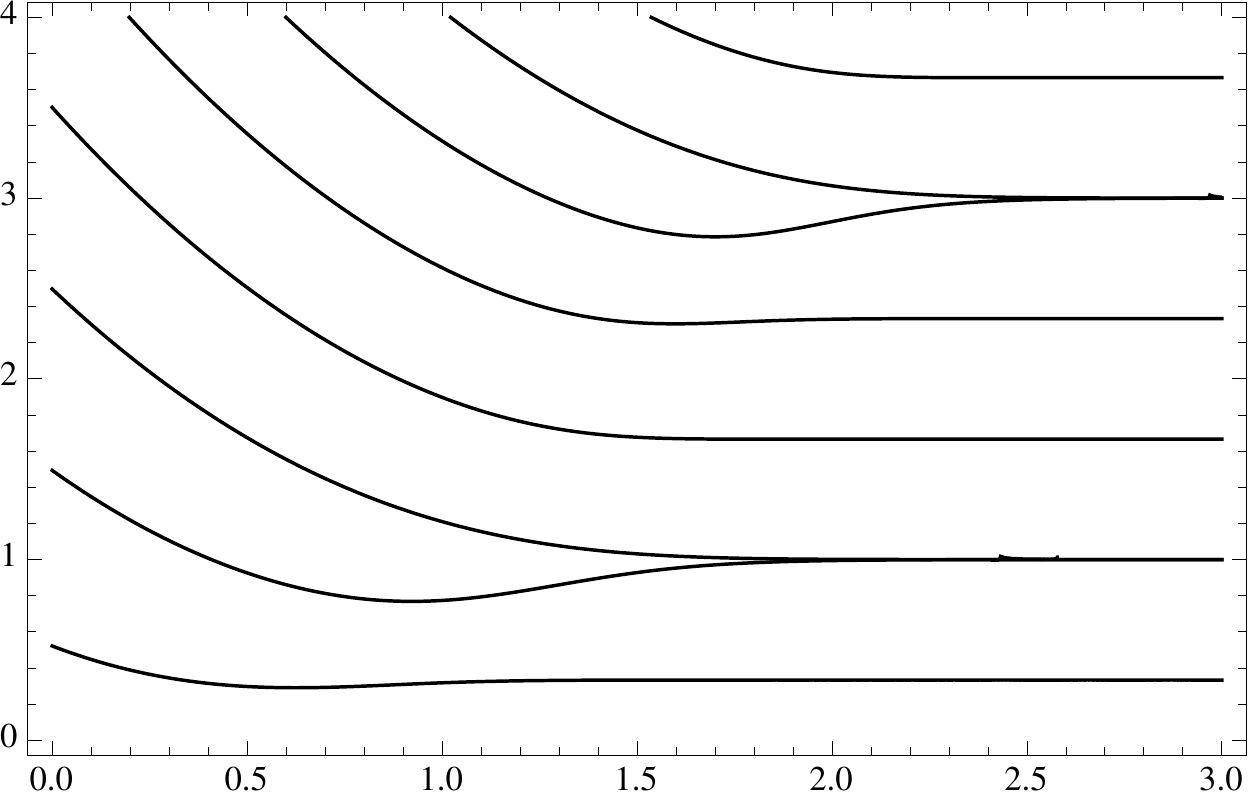}
\caption{Band functions $k\mapsto \lambda_{\gb,n}(k)$ for $\gb=(-\frac{1}{3},1)$.}
\label{F:2}
\end{figure}

\begin{figure}[htbp]
\centering
\includegraphics[width=10cm]{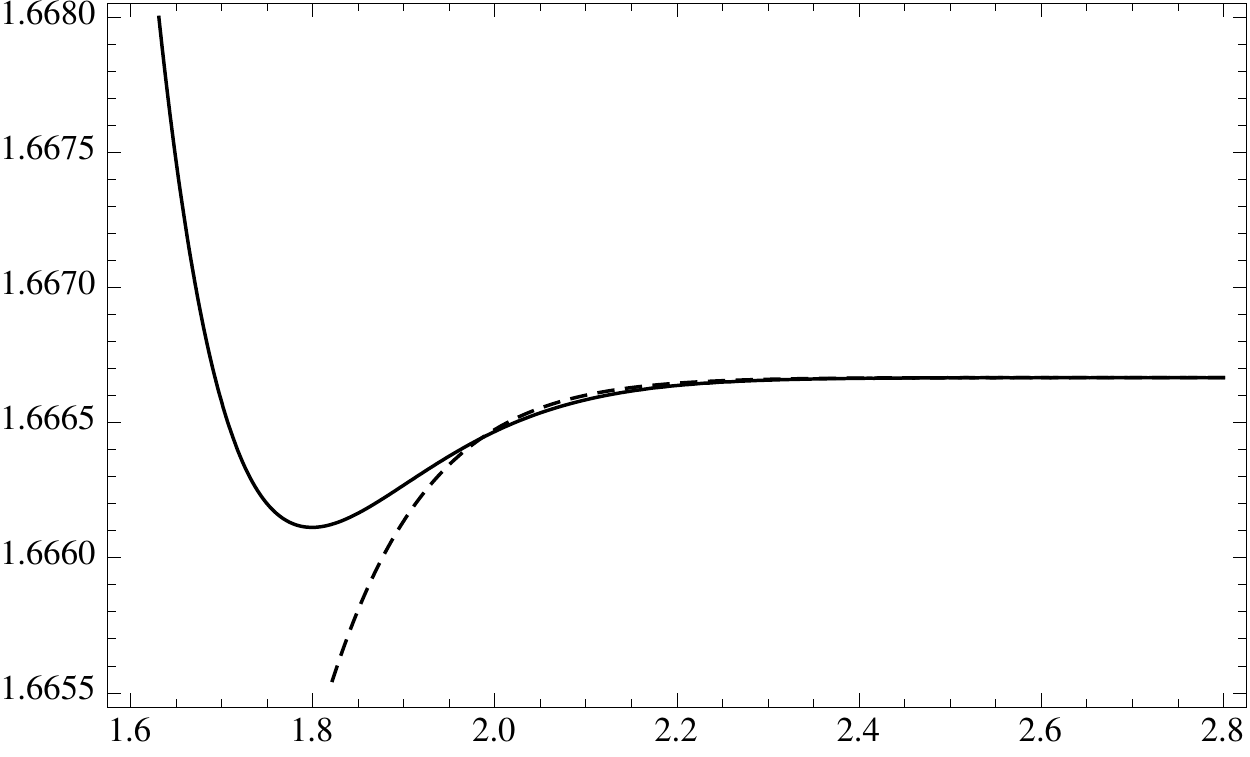}
\caption{Zoom near $\lambda=\frac{5}{3}$ on the band function $k\mapsto \lambda_{\gb,4}(k)$ for $\gb=(-\frac{1}{3},1)$. In dash line: the asymptotics from Therorem \ref{magnetic-step}, case (ib). Here $n=3$.}
\label{F:3}
\end{figure}

\begin{figure}[htbp]
\centering
\includegraphics[width=10cm]{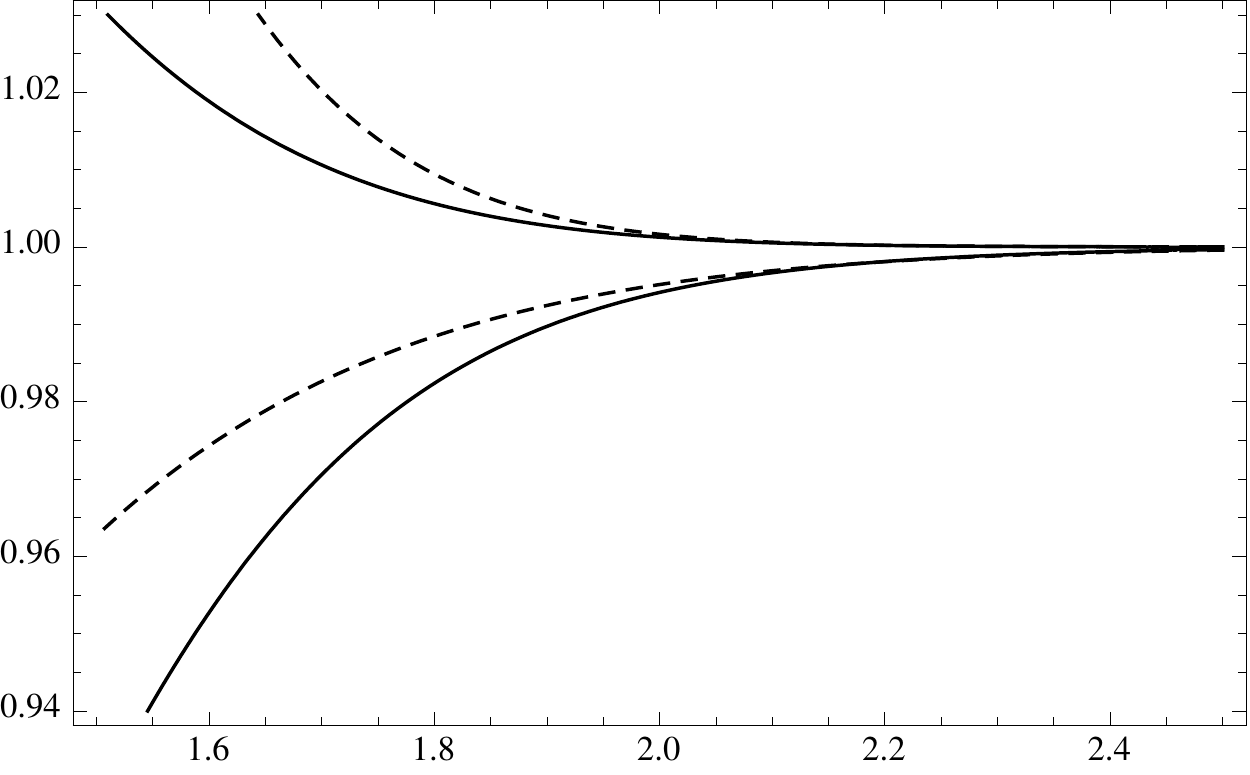}
\caption{Zoom near $\lambda=1$ on the band functions $k\mapsto \lambda_{\gb,p}(k)$ for $\gb=(-\frac{1}{3},1)$ and $p=2,3$ (splitting case). In dash line: the asymptotics from Therorem \ref{magnetic-step}, case (ii). Here $\lambda=1$, $n=1$ and $m=2$.}
\label{F:4}
\end{figure}

\newpage

\subsection{Consequence of the analysis of the band functions}
In this section we qualitatively present possible consequences of our results.
\paragraph{Standard perturbation of the Hamiltonian}
If we perturb the Hamiltonian $\mathcal{L}_{\gb}$ by a suitable electric potential, discrete spectrum might appear under the essential spectrum, which is the infimum of the first band function. In particular, the existence of a minimum is useful if one wants to study the counting function of created eigenvalues. Here we know that the infimum is reached for some finite frequency $k^{*}$. Then, following the standard procedure described for example in \cite{Rai92}, the number of discrete eigenvalues can be described, depending on the degeneracy of the minimum and on the decay of the electric perturbation.

In the same way, it is possible to derive a limiting absorption principle at energies close to the bottom of the spectrum of $\mathcal{L}_{\gb}$ (and also near each band minimum), as in \cite{Soc01}. The case of thresholds corresponding to the limits of band functions can also be studied in the spirit of~\cite{PofSoc14}. In particular the asymptotics of the band functions for large $k$ allow to describe the suitable functional spaces when writing a limiting absorption principle near the Landau levels.

\paragraph{Semiclassical magnetic Laplacian with piecewise constant magnetic field}

Let us consider the semiclassical magnetic Laplacian $(-i\hbar\nabla+\bA)^2$ in the plane with a magnetic field such that
\[
\nabla\times\bA=\bB=b_{1}\mathds{1}_{\Omega}+b_{2}\mathds{1}_{\complement\Omega},
\]
where $\Omega$ is a smooth and bounded domain. It is well-known that the semiclassical analysis of the low lying spectrum (motivated by the analysis of the third critical field in the Ginzburg--Landau theory) is governed by the spectral behavior of some local models (see \cite{FouHel10, BDP14}). When the magnetic field is not continuous, models operators in the form of $\mathcal{L}_{\gb}$ have to be considered. It is known that the minima (which are often non degenerate) of the corresponding band functions may combine with the curvature (of the boundary for instance) to generate the semiclassical asymptotics. In particular, the existence of a minimum of the band functions implies a microlocalization of the magnetic eigenfunctions (see \cite{FouHel10, Ray14, BHR14}).

\paragraph{Bulk states localized in energy near the thresholds}
In general Iwatsuka Hamiltonians (\textit{i.e.} Hamiltonians associated with magnetic fields depending only on one variable), the analysis of quantum states submitted to the magnetic field and localized in energy in an interval $I$ depends on the behavior of the band functions when crossing the energy interval $I$, see~\cite{ManPur97}. For example if $I$ does not contain any threshold, then the quantum states are localized in space near the variations of the magnetic field (here, the jump line $\{x=0\}$) and their current (defined below) is not small. This corresponds to the fact that a classical charged particle moves along the edge $\{x=0\}$, in a \enquote{snake orbit}.

One may also show, as in \cite{dBiePu99,HisPofSoc14}, that there exist states, called bulk states, localized in energy near the thresholds, bearing small current. If the threshold corresponds to a limit of a band function, these bulk states are exponentially small near the jump line. The precise (de)localization of these states and the current they bear can be derived from the asymptotics of the band function by the following method (see \cite{HisPofSoc14}).  Assume that $I=(\lambda-\delta,\lambda+\delta)$ where $\lambda\in \mathfrak{S}$ and $\delta\downarrow 0$. The current of a state $\psi\in \sL^2(\R)$ along the jump line is given by $\langle J_{y}\psi,\psi \rangle$ where $J_{y}:=\frac{i}{2} [\mathcal{L}_{\gb},y]=-i\partial_{y}-a_{\gb}(x)$ is the current operator. This quantity may be linked to the derivative of the band functions by using the Feynman--Hellmann formula. Using the asymptotics from Theorem \ref{magnetic-step}, one may show that the derivative of the band functions is small on the large $k$ interval corresponding to the values where $\lambda_{n}(k) \in I$. This would provide an upper bound on the current carried by bulk states, depending on how far from the threshold is their energy. Then follows an analysis in phase space of these states in order to show that delocalization in the frequency $k$ involves delocalization in the $x$-variable.

Note finally that the sign of the derivative of the band function is related to the direction of propagation of the particle along the jump line of the magnetic field. In particular quantum states related to non-monotonic band functions (see Corollary \ref{C:non-mono}) may flow along the line in two possible directions. This corresponds to the possible \emph{negative velocity} described in \cite{PeetRej}. We refer to~\cite{Yaf08} for an analysis of the direction of propagation of quantum states in a translation invariant magnetic field.


\section{Magnetic wall}\label{sec:mag-wall1}

This section is devoted to the case $\gb=(0,1)$. Away from the edge at $x=0$, a classical electron moves in a straight line for $x < 0$
and in a closed circular orbit for $x > 0$. If the orbit intersects the edge, then the classical electron escapes to $x = - \infty$.
The absence of magnetic field for $x<0$ results in a half-line of essential spectrum for the quantum electron.
As we will see, the spectrum varies with the sign of $k$.

The reduced potential $V_{\gb}$ satisfies $V_{\gb}(x,k)=k^2$ for $x\leq 0$, and 
$V_{\gb}(x,k)\to+\infty$ as $x\to+\infty$.
As a consequence, we easily deduce the following proposition.
\begin{proposition}
Assume that $\gb=(0,1)$ and $k\in\R$. The essential spectrum of $\gh_{b}(k)$ is given by
\[
\gS_{\ess}(\gh_{b}(k))=[k^2,+\infty).
\]
Moreover, $\gS(\mathcal{L}_{\gb})=[0,+\infty)$.
\end{proposition}
In fact, we can prove slightly more.
\begin{proposition}
Assume that $\gb=(0,1)$ and $k\in\R$. The operator $\gh_{\gb}(k)$ has no embedded eigenvalues in its essential spectrum.
\end{proposition}
\begin{proof}
Let us consider $\lambda\geq k^2$ and $\psi\in \Dom(\gh_{\gb}(k))$ such that:
\begin{equation}\label{eve}
-\psi''+(k-a_{\gb}(x))^2\psi=\lambda \psi.
\end{equation}
For $x<0$, we have $-\psi''=(\lambda-k^2) \psi$ whose only solution in $\sL^2(\R_{-})$ is zero. But since the solutions of \eqref{eve} belongs to $\sH^2_{\loc}$ and are in $\mathcal{C}^1(\R)$, this implies that $\psi(0)=\psi'(0)=0$ and thus $\psi\equiv 0$ by the Cauchy--Lipschitz theorem.
\end{proof}
Let us now describe the discrete spectrum, that is the eigenvalues $\lambda<k^2$. Since, for $k\leq 0$, we have $\gq_{\gb}(k)\geq k^2$, we deduce the following proposition by the min-max principle.
\begin{proposition}
Assume that $\gb=(0,1)$ and $k\leq 0$. Then
\[
\gS(\gh_{\gb}(k))=\gS_{\ess}(\gh_{\gb}(k))=[k^2,+\infty).
\]
\end{proposition}
Therefore we must only analyze the case when $k>0$. The following lemma is a reformulation of the eigenvalue problem.
\begin{lemma}\label{lemma-trans}
Assume that $\gb=(0,1)$ and $k>0$. The number $\lambda<k^2$ is an eigenvalue of $\gh_{\gb}(k)$ if and only if
there exists a non zero function $\psi\in\sL^2(\R_{+})$ satisfying
\begin{subnumcases}{}
-\psi''(x)+(x-k)^2\psi=\lambda \psi(x),\\
\psi'(0)-\sqrt{k^2-\lambda}\, \psi(0)=0\, .
\end{subnumcases}
Moreover, the eigenfunctions of $\gh_{\gb}(k)$ can only vanish on $\R_{+}$. The eigenvalues of $\gh_{\gb}(k)$ are simple.
\end{lemma}
\begin{proof}
We consider $\gh_{\gb}(k)\psi=\lambda\psi$. Since $\psi\in\sL^2(\R_{-})$, there exists $A\in\R$ such that, for $x\leq 0$, $\psi(x)=Ae^{x\sqrt{k^2-\lambda}}$. Then we have to solve $-\psi''+V_{\gb}(x,k)\psi=\lambda\psi$ for $x\geq 0$ with the transmission conditions
\[
\psi(0)=A \ \ \text{and} \ \ \psi'(0)=A\sqrt{k^2-\lambda} \,
\]
or equivalently
\[
\psi'(0)-\sqrt{k^2-\lambda}\,\psi(0)=0.
\]
In particular, $A$ cannot be zero. The simplicity is a consequence of the Cauchy--Lipschitz theorem.
\end{proof}

\begin{lemma}\label{non-decreasing}
Assume that $\gb=(0,1)$. The functions $\R_{+}\ni k\mapsto \lambda_{\gb,n}(k)$ are non decreasing.
\end{lemma}
\begin{proof}
We use the translation $x=y+k$ to see that $\gh_{b}(k)$ is unitarily equivalent to $D_{y}^2+\tilde V(y,k)$ with $\tilde V(y,k)=\one_{(-\infty,-k)}(y)k^2+\one_{(-k,+\infty)}(y)y^2$. For $0<k_{1}<k_{2}$, we have:
\begin{align*}
\tilde V(y,k_{2})-\tilde V(y,k_{1})&=\one_{(-k_{2},-k_{1})}(y)y^2+\one_{(-\infty,-k_{2})}(y)k_{2}^2-\one_{(-\infty,-k_{1})}(y)k_{1}^2\\
&\geq \one_{(-k_{2},-k_{1})}(y)k_{1}^2+\one_{(-\infty,-k_{2})}(y)k_{2}^2-\one_{(-\infty,-k_{1})}(y)k_{1}^2\\
&=\one_{(-\infty,-k_{2})}(y)k_{2}^2-\one_{(-\infty,-k_{2})}(y)k_{1}^2\geq 0.
\end{align*}
By the min-max principle, we infer the desired monotonicity.
\end{proof}
The next lemma is a consequence of the Sturm--Liouville theory (see for instance \cite{Si05}).
\begin{lemma}
\label{L:sturmL}
Assume that $\gb=(0,1)$, and let $n\in\N$ be such that $\lambda_{\gb,n}(k)<k^2$. Then, the corresponding eigenspace is one dimensional and is generated by a normalized function $\psi_{\gb,n}(k)$, depending analytically on $k$. Moreover, $\psi_{\gb,n}(k)$ has exactly $n-1$ zeros, all located on the positive half line.
\end{lemma}
\begin{proof}
We have only to explain the part of the statement concerned with the zeros. Thanks to Lemma \ref{lemma-trans}, one knows that the zeros are necessarily positive. Let us briefly recall the main lines of the Sturm's oscillation theorem. Let us denote by $m-1$ the number of zeros of $\psi=\psi_{\gb,n}(k)$ and let us check that $m\leq n$. For $j=1,\ldots,m-1$, $z_{n,j}$ denotes the $j$-th zero of $\psi_{\gb,n}(k)$  (the zeros are simple and isolated by the Cauchy--Lipschitz theorem) and we let $\varphi_{j}=\psi_{|(z_{n,j}, z_{n,j+1})}$ extended by zero. By definition,
\[
\gq_{\gb}(k)(\varphi_{j},\varphi_{\ell})=\int_{\R} \varphi_{j}' \varphi_{\ell}'+V_{\gb}(x,k) \varphi_{j} \varphi_{\ell} \dx x,
\]
so,
\[
\gq_{\gb}(k)(\varphi_{j},\varphi_{\ell})=\delta_{j,\ell} \int_{\R} \varphi_{j}'^2+V_{\gb}(x,k) \varphi_{j}^2 \dx x=\lambda_{\gb,n}\delta_{j,\ell}\|\varphi_{j}\|^2.
\]
We apply the min-max principle to the space $F_{m}=\underset{1\leq j\leq m}{\spann}\varphi_{j}$ to get $\lambda_{\gb,m}\leq \lambda_{\gb,n}$ and thus $m\leq n$. In particular we recover that $\psi_{\gb,1}$ does not vanish. Let us explain why $\psi_{\gb, 2}$ vanishes exactly once. Introducing the Wronskian $W_{1,2}=\psi_{1}\psi'_{2}-\psi_{2}\psi_{1}'$, we notice that $W'_{1,2}=(\lambda_{\gb,1}-\lambda_{\gb,2})\psi_{\gb,1}\psi_{\gb,2}$. If $\psi_{\gb,2}>0$, then $W_{1,2}'<0$. Since the eigenfunctions belong to the Schwartz class at infinity, $W_{1,2}(x)\to0$ as $|x|\to+\infty$. This is a contradiction. With the same kind of monotonicity arguments, we can prove that the zeros of the eigenfunctions are interlaced and that $\psi_{\gb,n}$ has exactly $n-1$ zeros.
\end{proof}

By using the harmonic approximation in the semiclassical limit (see \cite{DiSj99} and the Section \ref{subsec:band-fncs1} below), we can prove the following lemma.
\begin{lemma}\label{limit}
Assume that $\gb=(0,1)$.
For all $n\geq 1$, we have
\[
\lim_{k\to+\infty}\lambda_{\gb,n}(k)=E_{n}.
\]
In particular, for $k$ large enough, we have $\lambda_{\gb,n}(k)\leq E_{n}<k^2$.
\end{lemma}
Let us now prove that the $n$-th band function lies between the two consecutive Landau levels $E_{n-1}$ and $E_{n}$.
\begin{proposition}\label{strip}
Assume that $\gb=(0,1)$.
For all $n\geq 1$ and for all $k>0$ such that $\lambda_{\gb,n}(k)<k^2$ we have $\lambda_{\gb,n}(k)\in\left(E_{n-1},E_{n}\right)$.
\end{proposition}
\begin{proof}
By Lemmas~\ref{non-decreasing} and~\ref{limit}, $\lambda_{\gb,n}(k)<E_{n}$ (the strict inequality comes from the analyticity). It remains to prove that $\lambda_{\gb,n}(k)>E_{n-1}$. Clearly $\lambda_{\gb,1}(k)>E_{0}$. We rely on Notation~\ref{ntn:theta1}. For $n\geq 2$, we consider the function $\varphi_{n}(x)=\psi_{\gb, n}(x+z_{n,1}(k))$ which satisfies $\gh^{\rm D}(k-z_{n,1}(k))\varphi_{n}=\lambda_{\gb,n}(k)\varphi_{n}$ and has exactly $n-2$ zeros on $\R_{+}$, cf. Lemma~\ref{L:sturmL}. By the Sturm's oscillation theorem, $\varphi_{n}$ is the $(n-1)$-th eigenfunction of $\gh^{\rm D}(k-z_{n,1}(k))$. Therefore we have $\lambda_{\gb,n}(k)=\mu^{\rm D}_{n-1}(k-z_{n,1}(k))$.
Moreover for all $\ell\geq 1$ and $k\in\R$, $\mu^{\rm D}_{\ell}(k)>E_{\ell}$.
This follows by domain-monotonicity, once we notice that $\gh^{\rm D}(k)$, after
a translation, is given by the harmonic oscillator $D_x^2+x^2$ on
$(-k,+\infty)$ with Dirichlet boundary condition at $x=-k$, together with
the fact that the harmonic oscillator $D_x^2+x^2$ on $\R$ has simple
eigenvalues $E_{\ell}$. This provides the desired conclusion.
\end{proof}

We recall that the sequence $(\xi_{n})_{n\geq1}$ is defined in Notation \ref{ntn:theta1}. For the next proposition, we will need to use the so-called parabolic cylinder functions. We denote by $U(a,x)$ the first Weber parabolic cylinder
function (whose properties are summarized in the Appendix, section \ref{appendix}; for more details see, for example, \cite[section 12]{DLMF} or \cite{AbSt64}) which is a solution of the linear ordinary differential equation:
\begin{equation}
\label{E:weber}
-y''(x)+\frac{1}{4}x^2 y(x)=-a y(x).
\end{equation}
The Weber function $U(a,x)$ decays exponentially when $x\to+\infty$.

\begin{proposition}
Assume that $\gb=(0,1)$. For all $n\geq 1$, the equation $\lambda_{\gb,n}(k)=k^2$ has a unique non negative solution,  $k=\xi_{n-1}$, such that, locally, for $k>k_{n}$, $\lambda_{\gb,n}(k)<k^2$. Moreover we have $\lambda_{\gb,n}(\xi_{n-1})=\Theta_{n-1}$.
\end{proposition}

\begin{proof}
Thanks to Lemma \ref{limit}, we can define $k_{n}=\max\{k\geq0 : \lambda_{\gb,n}(k)=k^2\}$. By continuity, we have, for all $k>k_{n}$, $\lambda_{\gb,n}(k)<k^2$. Let us now prove the uniqueness. Let us consider a solution $\tilde k_{n}\geq 0$. For all integer $p\geq p_{0}$ with $p_{0}$ large enough, we have $\lambda_{\gb,n}\left(\tilde k_{n}+\frac 1 p\right)<\left(\tilde k_{n}+\frac 1 p\right)^2$.
Let us now consider the eigenvalue equation
\begin{equation}\label{eve'}
D_{x}^2\varphi_{n,p}+(\tilde k_{n,p}-a_{\gb}(x))^2\varphi_{n,p}=\lambda_{n,p}\varphi_{n,p},
\end{equation}
where $\varphi_{n,p}=\psi_{\gb,n}(\tilde k_{n,p})$, $\tilde k_{n,p}=\tilde k_{n}+\frac{1}{p}$ and $\lambda_{n,p}=\lambda_{\gb,n}(\tilde k_{n,p})$. Let us investigate the limit $p\to+\infty$. As seen in the proof of Lemma \ref{lemma-trans}, we know that there exists $\alpha\in \R^{*}$ such that, for $x\leq 0$,
\[
\varphi_{n,p}(x)=\alpha e^{x\sqrt{\tilde k_{n,p}^2-\lambda_{n,p}}}.
\]
By solving~\eqref{eve'} on $x\geq0$ and using the parabolic cylinder function $U$, we find that there exits $\beta$ such that, for $x\geq 0$,
\[
\varphi_{n,p}(x)=\beta U\left(-\frac{\lambda_{n,p}}{2} ; \sqrt{2}(x-\tilde k_{n,p})  \right).
\]
Since $\varphi_{n,p}\neq 0$, we have $(\alpha,\beta)\neq (0,0)$ and $\varphi_{n,p}$ is $\mathcal{C}^1$ at $x=0$, we get the transmission condition:
\[
\sqrt{\tilde k_{n,p}^2-\lambda_{n,p}}\hat U\left(-\frac{\lambda_{n,p}}{2};-\tilde k_{n,p}\sqrt{2}\right)-\sqrt{2} U'\left(-\frac{\lambda_{n,p}}{2} ;-\tilde k_{n,p}\sqrt{2} \right)=0.
\]
 By continuity and taking the limit $p\to+\infty$, we get
\[
U'\left(-\frac{\tilde k_{n}^2}{2} ;-\tilde k_{n}\sqrt{2} \right)=0.
\]
Notice that the function $x\mapsto U\left(-\frac{\tilde k_{n}^2}{2}; \sqrt{2}(x-\tilde k_{n})\right)$ solves the differential equation
\begin{equation*}
\left\{
\begin{aligned}
& -y''(x)+(x-\tilde k_{n})^2y(x)=\tilde k_{n}^2 y(x)
\\
& y'(0)=0 \ \ \text{and} \ \ y(0) \neq 0
\end{aligned}
\right.
\end{equation*}
Moreover it belongs to $\mathcal{S}(\overline{\R_{+}})$. Therefore, there exists $\ell\geq 1$ such that $\mu^N_{\ell}(\tilde k_{n})=\tilde k_{n}^2$, and therefore $\tilde k_{n}=\xi_{\ell-1}$ and $\tilde k_{n}^2=\Theta_{\ell-1}$. By Proposition \ref{strip}, we know that $\tilde k_{n}^2=\lambda_{\gb,n}(\tilde k_{n})\in[E_{n-1},E_{n}]$. Moreover, we recall that $\Theta_{\ell-1}\in\left(E_{\ell-1},E_{\ell}\right)$. This implies that $\ell=n$.
\end{proof}

\section{Trapping magnetic step}\label{sec:trapping-steps1}

In this section, we focus on the case $\gb=(b_{1},1)$ with $-1<b_{1}<0$. The case of $b_1 = -1$ is treated in
\cite{DomHisSoc13}. In that case, due to the symmetry, all levels are split as $k \to +\infty$. The asymptotic expansion
comes from \cite{HisPofSoc14,Popoff}. The band functions for this case are shown in Figure~\ref{F:1}.

When $b_1 < 0$, the effective potential $V_{\gb}(x ; k)$, described in \eqref{eq:eff-potential1}, diverges to infinity as $x \to \pm \infty$, so that for all $k\in\R$, the operator $\gh_{\gb}(k)$ has compact resolvent. The spectrum of $\gh_{\gb}(k)$ is discrete and the eigenvalues $(\lambda_{\gb,n}(k))_{n\geq 1}$ are all simple.

For $k<0$, the effective potential $V_{\gb}(x ; k)$ has the form of an asymmetric potential well, see Figure \ref{fig:case2}.
It follows that $V_{\gb}(\cdot ; k)\geq k^2$ so that, for $k\leq 0$, $\lambda_{\gb,n}(k)\geq k^2$. In particular, we have, for all $n\geq 1$,
\[
\lim_{k\to-\infty}\lambda_{\gb,n}(k)=+\infty, ~~k < 0.
\]
For $k>0$, the effective potential
has the form of an asymmetric double-well:
\beq\label{eq:eff-potential3posk}
V_{\gb}(x ;k) = \left\{ \begin{array}{cc}
                            (k - b_1x)^2 & x < 0 \\
                            (k - x)^2 & x>0
                            \end{array}
                            \right.
\eeq
This asymmetric double-well admits two nondegenerate minima at $x=\frac{k}{b_{1}}$ and $x=k$, for $k>0$ and $b_1 < 0$.

Let us now investigate the behavior of the band functions in the limit $k\to+\infty$ and prove Theorem \ref{magnetic-step}.
Due to the asymmetry of the double-well potential \eqref{eq:eff-potential3posk}, we will prove, among other results, that not all the Landau levels exhibit an exponentially small splitting as in the case of $(b_1,b_2) = (-1,1)$.

\subsection{Limits of the band functions}\label{subsec:band-fncs1}

Let us first describe the set of the limits of the band functions $\lambda_{\gb,n}(k)$ as $k \rightarrow + \infty$. As noted above, the limit as $k \rightarrow - \infty$ of any band function is plus infinity.


\begin{lemma}\label{set-of-limits}
Assume that $\gb=(b_1,1)$, with $-1<b_1<0$. Recall from Definition \ref{limit} that $\mathfrak{L}$ denotes the union of the Landau levels for magnetic fields $b_1\in(-1,0)$ and $b_2 = 1$, and that $\mathfrak{S}$ denotes the splitting set $\{ E_N, n\geq 1\} \cap \{ |b_1| E_n, n \geq 1 \}$.Then the set  of limits as $k \rightarrow + \infty$ of the band functions is precisely the set of Landau levels $\mathcal{L}$:
\[
\Bigl\{\lim_{k\to+\infty}\lambda_{\gb,n}(k), n\geq 1\Bigr\}=\mathfrak{L}.
\]
The band functions exhibit the following behavior as $k \rightarrow + \infty$:
\begin{enumerate}
\item Non-splitting case. If a band function limit $\lambda\in\mathfrak{L}\setminus\mathfrak{S}$,
then for any $\epsilon_0 > 0$,
there exists a $K>0$ such that if $k\geq K$ then it holds that $\rank\mathds{1}_{[\lambda-\eps_{0},\lambda+\eps_{0}]}(\gh_{\gb}(k))=1$ and there exists a unique $n\geq 1$ such that $\lambda_{\gb,n}$ converges towards $\lambda$.

\item Splitting case. If a band function limit $\lambda\in\mathfrak{S}$,
then for any $\epsilon_0 > 0$,
there exists a $K>0$ such that if $k\geq K$ then it holds that $\rank\mathds{1}_{[\lambda-\eps_{0},\lambda+\eps_{0}]}(\gh_{\gb}(k))=2$ and there exists a unique $n\geq 1$ such that $\lambda_{\gb,n}$ and $\lambda_{\gb,n+1}$ converge towards $\lambda$.
\end{enumerate}
\end{lemma}

\begin{proof}
Let us first provide a semiclassical reformulation and summarize the standard procedure of the harmonic approximation for multiple wells (see \cite[Chapter 4]{He88} for example). For $k>0$, we let $x=ky$ and the operator $\gh_{\gb}(k)$ is unitarily equivalent to $k^2\tilde\gh_{\gb,k^{-2}}$, where $\tilde\gh_{\gb,\hbar}$ is the self-adjoint realization on $\sL^2(\R)$ of the differential operator $\hbar^2 D_{y}^2+V_{\gb}(y ; 1)$. The minima of $V_{\gb}(\cdot ; 1)$ are reached for $x=b_{1}^{-1}<0$ and $x=1$. We have $V_{\gb}''(1; 1)=2$ and $V_{\gb}''(b^{-1}_{1} ; 1)=2|b_{1}|$. Let us just recall the spirit of the analysis. Let us consider $\eta>0$ such that $b_{1}^{-1}+\eta<0$ and $1-\eta>0$. Then, let us consider the operator $\tilde\gh_{\gb,\hbar,\eta}=\tilde\gh^\lef_{\gb,\hbar,\eta}\oplus \tilde\gh^\rig_{\gb,\hbar,\eta}$ where $\tilde\gh^\lef_{\gb,\hbar}$ and $\tilde\gh^\rig_{\gb,\hbar}$ are the Dirichlet realizations of $\hbar^2D_{y}^2+V(y; 1)$ on $\sL^2((-\infty,b_{1}+\eta))$ and on $\sL^2((1-\eta,+\infty))$ respectively. Let us now consider an arbitrary constant $C>0$. It is well-known (see Helffer's lecture notes \cite{He88}) that any eigenfunction of $\tilde\gh_{\gb,\hbar,\eta}$ or of  $\tilde\gh_{\gb,\hbar}$  associated with an eigenvalue $\lambda\leq C\hbar$ satisfies decay estimates of Agmon type away from the minima. In particular, this allows to prove the existence of $c>0$ such that $\gS(\tilde\gh_{\gb,\hbar})\cap\{\lambda\in\R : \lambda\leq C\hbar\}$ and $\gS(\tilde\gh_{\gb,\hbar,\eta})\cap\{\lambda\in\R : \lambda\leq C\hbar\}$ coincides modulo $\mathcal{O}\bigl(e^{-c\hbar^{-1}}\bigr)$. Then, by the harmonic approximation, the spectrum of $\tilde\gh_{\gb,\hbar,\eta}$ is well-known in the semiclassical limit: the $n$-th eigenvalue of $\tilde\gh^\lef_{\gb,\hbar,\eta}$ is equivalent to $|b_{1}|E_{n}\hbar$ and the one of $\tilde\gh^\rig_{\gb,\hbar,\eta}$ to $E_{n}\hbar$. The desired conclusion follows. The last part of the statement follows from the fact that the spectra (with multiplicity) of $\tilde\gh_{\gb,\hbar}$ and $\tilde\gh_{\gb,\hbar,\eta}$ coincide modulo an exponentially small error.
\end{proof}

As described in \cite{HisPofSoc14}, the convergence of the band functions toward these limits corresponds to the existence of bulk states with energies close to the Landau levels. We next provide a more accurate description of the convergence of the band functions towards their limits. This is important for two reasons. First, this allows a more quantitative characterization of the associated bulk states following the analysis of \cite{HisPofSoc14}. Second, the manner in which the band function approaches its limit as $k \rightarrow + \infty$, that is, whether it approaches from above or below the Landau level, provides a hint as to whether the band function is
globally monotonic or not. In order to describe the asymptotic behavior of the band function, we will leave the semiclassical techniques and use the characterization of the eigenvalues $\lambda_{\gb,n}(k)$ as zeros of certain special functions.

\subsection{Splitting and non-splitting of levels}\label{subsec:split-or-not1}

We consider $k>0$ and calculate the behavior of the energy levels in the splitting and the non-splitting cases.
This is a refinement of Lemma \ref{set-of-limits}, where we have shown that each band function $\lambda_{\gb,n}$ tends to a threshold
that is an element of the set of Landau levels $\mathfrak{L}$. There is no explicit formula for the index of the corresponding Landau level as a function of $n$, the band index.

In this section, for the sake of simplicity, we will consider a converse strategy: Given a threshold $E_{n}$ or $bE_{n}$, where $n \in \N^{*}$ is fixed, we describe the asymptotic behavior of the band functions converging toward this threshold. As described in Lemma \ref{set-of-limits}, there are either one or two band functions $\lambda_{\gb,p_n} (k)$ converging toward a threshold
$E_{n}$ or $bE_{n}$ in $\mathfrak{L}$.
Although the indices $p_n$ of these band functions are not explicit in terms of the index $n$, the asymptotics of the related band functions $\lambda_{\gb,p_n} (k)$, as $k \rightarrow + \infty$,
are explicit functions of $n$.

Let us notice that this procedure will provide the asymptotics of all the band functions since there is a bijection between their limits and the set of the thresholds $\mathfrak{L}$ (counted with multiplicity).

\subsubsection{Algebraic characterization of the eigenvalues of $\gh_{\gb}(k)$}
 We begin with an algebraic characterization of the eigenvalues of $\gh_{\gb}(k)$. Recall that $U(a,x)$ is the first Weber parabolic cylindrical function introduced in Section \ref{sec:mag-wall1} as a solution of \eqref{E:weber}.



\begin{proposition}\label{prop:determinant1}
Assume that $\gb=(b_1,1)$, with $-1<b_1<0$.
A number $\lambda$ is an eigenvalue of $\gh_{\gb}(k)$ for $k \in \R$ if and only if it is a solution of $D_{\gb}(\lambda,k)=0$, where
\[
\begin{aligned}
D_{\gb}(\lambda,k)&:=U(-\tfrac{\lambda}{2b},-\sqrt{2}\tfrac{k}{\sqrt{b}})
U'(-\tfrac{\lambda}{2},-\sqrt{2}k)\\
&\quad+\sqrt{b}U'(-\tfrac{\lambda}{2b},-\sqrt{2}\tfrac{k}{\sqrt{b}})U(-\tfrac{\lambda}{2},-\sqrt{2}k)
\end{aligned}
\]
and $b=|b_{1}|$.
\end{proposition}

\begin{proof}
From the eigenvalue equation and the form of the potential \eqref{eq:eff-potential3posk}, we are led to solve for $x>0$:
\[
-y''(x)+(x-k)^2y(x)=\lambda y(x)
\]
with the condition that the solution is in $\sL^2(\R_{+})$. The only such solution is
\[
x\mapsto U(-\tfrac{\lambda}{2},\sqrt{2}(x-k)).
\]
In particular, when $\lambda=2n-1$ is a Landau level, the solution is the $n$-th Hermite function centered at $x=k$.
On the half line $x < 0$, we solve
\[
-y''(x)+(b_{1}x-k)^2y(x)=\lambda y(x)
\]
with the condition that the solution is in $\sL^2(\R_{-})$. This solution may be expressed as
\[
x\mapsto U(-\tfrac{\lambda}{2b},\sqrt{2}(-\sqrt{b}x-\tfrac{k}{\sqrt{b}})).
\]
Therefore, if $\psi_{\gb,n}(\cdot,k)$ is an eigenfunction associated with $\lambda=\lambda_{\gb, n}(k)$, it may be written as (up to an
overall multiplicative constant):
\[
\psi_{\gb, n}(x,k)=
\left\{
\begin{aligned}
&U(-\tfrac{\lambda}{2b},\sqrt{2}(-\sqrt{b}x-\tfrac{k}{\sqrt{b}})) \quad \mbox{if} \ \ x<0
\\
&\alpha U(-\tfrac{\lambda}{2},\sqrt{2}(x-k))  \quad \mbox{if} \ \ x>0
\end{aligned}
\right.
\]
where $\alpha\in \R$ is a constant. Since the eigenfunction is $\mathcal{C}^1$ at 0, we have
\[
\left\{
\begin{aligned}
&U(-\tfrac{\lambda}{2b},-\sqrt{2}\tfrac{k}{\sqrt{b}})=\alpha U(-\tfrac{\lambda}{2},-\sqrt{2}k)
\\
&-\sqrt{2b} U'(-\tfrac{\lambda}{2b},-\sqrt{2}\tfrac{k}{\sqrt{b}})= \alpha\sqrt{2}U'(-\tfrac{\lambda}{2},-\sqrt{2}k)
\end{aligned}
\right.
\]
It is possible to solve the above conditions if and only if the $2\times2$ matrix
\[
\begin{pmatrix}
U(-\tfrac{\lambda}{2b},-\sqrt{2}\tfrac{k}{\sqrt{b}}) & U(-\tfrac{\lambda}{2},-\sqrt{2}k)
\\
-\sqrt{b}U'(-\tfrac{\lambda}{2b},-\sqrt{2}\tfrac{k}{\sqrt{b}}) & U'(-\tfrac{\lambda}{2},-\sqrt{2}k)
\end{pmatrix}
\]
is non-invertible.
\end{proof}

We will use the following result on the asymptotic expansion
of the determinant $D_{\gb}(\lambda,k)$ as $k \to + \infty$.

\begin{lemma}\label{Tj} Assume that $\gb=(b_1,1)$, with $-1<b_1<0$, and let
$b= |b_1|$. The determinant $D_{\gb}(\lambda,k)$ of Proposition \ref{prop:determinant1}
can be written as
\begin{equation}
\label{E:Devwronskien}
\begin{aligned}
D_{\gb}(\lambda,k)
&=e^{(1+\frac{1}{b})\frac{k^2}{2}}T_{1}(\lambda,k)
+e^{(-1+\frac{1}{b})\frac{k^2}{2}}T_{2}(\lambda,k)\\
&\quad +e^{(1-\frac{1}{b})\frac{k^2}{2}}T_{3}(\lambda,k)
+e^{(-1-\frac{1}{b})\frac{k^2}{2}}T_{4}(\lambda,k),
\end{aligned}
\end{equation}
where the factors $T_{j}(\lambda,k)$ admit the following asymptotic expansions as $k\to+\infty$:
\[
\left\{
\begin{aligned}
T_{1}(\lambda,k)&=\frac{1}{\Gamma(\frac{b-\lambda}{2b})\Gamma(\frac{1-\lambda}{2})}k^{-\frac{b+1}{2b}\lambda}2^{\frac{3}{2}-\frac{(1+b)\lambda}{4b}}b^{\frac{b+\lambda}{4b}}\pi\left(1+\mathcal{O}(k^{-2})\right),
\\
T_{2}(\lambda,k)&=\frac{1}{\Gamma(\frac{b-\lambda}{2b})} k^{\frac{b-1}{2b}\lambda-2}
\left(\frac{\sqrt{\pi}}{2} \right) 2^{\lambda\frac{b-1}{4b}}b^{\frac{b+\lambda}{4b}}(1+b)(-\sin\frac{\lambda\pi}{2})\left(1+\mathcal{O}(k^{-2})\right),\\
T_{3}(\lambda,k)&=\frac{1}{\Gamma(\frac{1-\lambda}{2})}k^{\frac{-b+1}{2b}\lambda-2} \left(\frac{\sqrt{\pi}}{2}\right)2^{\lambda\frac{-b+1}{4b}}b^{\frac{b-\lambda}{4b}}(1+b)(-\sin
\frac{\lambda\pi}{2b})\left(1+\mathcal{O}(k^{-2})\right),\\
T_{4}(\lambda,k)&=\sqrt{2}k^{\frac{b+1}{2b}\lambda}2^{\lambda\frac{b+1}{4b}}b^{\frac{b-\lambda}{4b}}\cos\frac{(1+b)\lambda\pi}{2b}\left(1+\mathcal{O}(k^{-2})\right).
\end{aligned}
\right.
\]
\end{lemma}

\begin{proof}
These expansions follow by the asymptotic expansion of the Weber
function $U$. We refer the appendix, section \ref{appendix}, and to~\cite[Chapter~19]{AbSt64}.
\end{proof}

\begin{remark}
We have controlled these rather tedious calculations with the computer software
\emph{Mathematica}.
\end{remark}

\begin{remark}
To obtain the large $k$ asymptotics of $D_{\gb}(\lambda,k)$, we factor out the leading exponential factor $e^{(1+\frac{1}{b})\frac{k^2}{2}}$ from the right side of \eqref{E:Devwronskien}.
The zeros of $D_{\gb}(\lambda,k)$ are the same as the zeros of
$$
\tilde{D}_{\gb}(\lambda,k) := e^{-(1+\frac{1}{b})\frac{k^2}{2}}D_{\gb}(\lambda,k),
$$
where
\begin{equation}
\label{E:Devwronskien-2}
\tilde{D}_{\gb}(\lambda,k) = T_{1}(\lambda,k) + e^{-{k^2}}T_{2}(\lambda,k)
+e^{-\frac{k^2}{b}}T_{3}(\lambda,k) + e^{-(1+\frac{1}{b}){k^2}}T_{4}(\lambda,k) .
\end{equation}
We next note that the functions $T_1$ and $T_2$ are bounded as $k \rightarrow + \infty$.
The coefficient $T_3$ grows like $k^{\left( \frac{1-b}{2b} \right) \lambda - 2}$ for $\lambda$ large enough.
Similarly, the term $T_4$ is increasing in $k$. However, the prefactors of these terms exhibit Gaussian decay.

The strategy of the proof of Theorem \ref{magnetic-step} consists of writing the band function $\lambda_{\gb,p_n} (k)$
as
$$
\lambda_{\gb,p_n} (k) = \tilde{E}_n + \eps_n (k),
$$
where $\lim_{k \rightarrow + \infty} \lambda_{\gb,p_n} (k) = \tilde{E}_n$, with either $\tilde{E}_n = E_n$ or $\tilde{E}_n = b E_n$,
and $\lim_{k \rightarrow + \infty} \eps_n(k) = 0$. For all $k \in \R$, we have
$$
\tilde{D}_{\gb}(\lambda_{\gb,p_n} (k), k) = \tilde{D}_{\gb}(\tilde{E}_n + \eps_n (k),k) = 0 .
$$

The distinction between the non-splitting and splitting cases is encoded in the gamma function coefficients of the $T_j(\lambda,k)$ terms
in Lemma \ref{Tj}.

In the \textbf{non-splitting case}, we have for one of the gamma functions
$$
\Gamma \left( \frac{b - \lambda_{\gb,p_n} (k) }{2b} \right) \rightarrow \Gamma \left( \frac{b - \tilde{E}_n }{2b} \right).
$$
In case \eqref{1}, the argument of the gamma function is $((b+1)/2 - n)/b$ that is never a negative integer so $\Gamma^{-1}$ does not have a zero and does not contribute to the asymptotics. In case \eqref{2}, the argument of this gamma functions is
$(1-n)$ so there are poles of the gamma function. The corresponding zeros of $\Gamma^{-1}$ contribute to the vanishing rate of $T_1$ and $T_2$. As for the other gamma function, we have
$$
\Gamma \left( \frac{1 - \lambda_{\gb,p_n} (k) }{2} \right) \rightarrow
\left\{ \begin{array}{cc}
\Gamma \left( 1-n \right) & \eqref{1} \\
\Gamma \left( \frac{b +1}{2} - nb \right) & \eqref{2}
\end{array}
\right.,
$$
so there is a pole in case \eqref{1}. These asymptotics are developed in Section \ref{subsubsec:no-splitting1}.

In the \textbf{splitting case}, we have that $b = E_n/E_m$, and it follows that both gamma functions have poles. Consequently, both factors $\Gamma^{-1}$ have a zero as $k \rightarrow + \infty$ and these contribute to the asymptotics as developed in Section \ref{subsubsec:splitting1}. This analysis also shows that $T_1(\tilde{E}_n,k) = 0$ since for \eqref{1} the inverse of the gamma function $\Gamma ( \frac{1 - \lambda}{2})$ vanishes for $\lambda = E_n$ and in the case \eqref{2} the inverse of the other gamma function $\Gamma ( \frac{b - \lambda }{2b} ) = 0$ for $\lambda = bE_n$.
\end{remark}

\subsubsection{No splitting of levels}\label{subsubsec:no-splitting1}

%
Let us prove Item \eqref{I} in Theorem \ref{magnetic-step}. The existence of $p_{n}$ comes from Lemma \ref{set-of-limits}. \Bk Let us assume case \eqref{1}, the proof of case \eqref{2} being symmetric. Then
\begin{equation}
\label{E:ExpandGamma1}
\lim_{k\to+\infty}\frac{1}{\Gamma(\tfrac{b-\lambda_{\gb,p_{n}}(k)}{2b})}=\frac{1}{\Gamma(\tfrac{b-E_{n}}{2b})}\neq0,
\end{equation}
the last relation coming from the facts that the only roots of $1/\Gamma$ are the negative integers and that $\frac{b-E_{n}}{2}\notin -\N$, since this is not a splitting case. Recall that the roots of $1/\Gamma$ are simple. In particular, as $k$ gets large,
\begin{equation}
\label{E:derive1sG}
\begin{aligned}
\frac{1}{\Gamma(\frac{1-\lambda_{\gb,p_{n}}(k)}{2})}
&=-\tfrac{\eps_{n}(k)}{2}
\left(\tfrac{1}{\Gamma}\right)'(1-n)(1+O(\eps_{n}(k))\\
&=(-\tfrac{\eps_{n}(k)}{2})(-1)^{n-1}(n-1)!(1+o(1)).
\end{aligned}
\end{equation}
As mentioned above, we replace $\lambda_{\gb,p_{n}}(k)$ by $E_{n}+\eps_{n}(k)$ in the equation $\tilde{D}_{\gb}(\lambda_{\gb,p_{n}}(k),k)=0$
and obtain an equation in the form
$$\hat\alpha(k) \eps_{n}(k)+\hat\beta(k)=0,$$
where $\hat\alpha$ and $\hat\beta$ are analytic functions such that
$$\hat\alpha(k)\underset{k\to+\infty}{\sim}-\frac{1}{2}(-1)^{n-1}(n-1)!\frac{1}{\Gamma(\frac{b-E_{n}}{2b})}k^{-\frac{b+1}{2b}E_{n}}2^{\frac{3}{2}-\frac{(1+b)E_{n}}{4b}}b^{\frac{b+E_{n}}{4b}}\pi,$$
$$\hat\beta(k)\underset{k\to+\infty}{\sim}-\frac{1}{\Gamma(\frac{b-E_{n}}{2b})} k^{\frac{b-1}{2b}E_{n}-2}
\left(\frac{\sqrt{\pi}}{2} \right) 2^{E_{n}\frac{b-1}{4b}}b^{\frac{b+E_{n}}{4b}}(1+b)\sin\left(\frac{E_{n}\pi}{2}\right)e^{-k^2},$$
obtained by using Lemma \ref{Tj}. We deduce the asymptotic behavior of $\eps_{n}(k)$ as $k\to+\infty$.

Since $\eps_n(k)$ is negative the corresponding band functions $\lambda_{\gb,n}(k)$ approach the Landau level $E_{n}$ or $bE_{n}$ as $k \rightarrow + \infty$ from below. Taking into account the fact that $\lambda_{\gb,p_{n}}(k) \rightarrow + \infty$ as $k \rightarrow - \infty$, this means that these band functions are not monotonic and have at least one minimum.

\subsubsection{Splitting of levels}\label{subsubsec:splitting1}
Let us finally give the proof of Item \eqref{II} in Theorem \ref{magnetic-step}. When the condition for splitting is achieved, two band functions have the same limit as $k \rightarrow + \infty$. We prove that one level approaches the common Landau level from below, and hence has a minimum, whereas the other band function tends to its limit by above.

The existence of $p_{n}$ is a consequence of Lemma \ref{set-of-limits}.
We also know from this lemma that both functions $\eps_{n}^{\pm}$ tends to 0 as $k\to+\infty$. We use the generic notation $\lambda(k)=E_{n}+\eps(k)$ satified by these two functions. Recalling that $b=\frac{E_{n}}{E_{m}}$, we note that as $k$ tends to $+\infty$:
\[
\frac{1}{\Gamma\left(\frac{b-\lambda(k)}{2b}\right)}=\frac{1}{\Gamma\left(1-m-\frac{\eps(k)}{2b}\right)}=-\frac{\eps(k)}{2b}\left(\frac{1}{\Gamma}\right)'(1-m)+\mathcal{O}(\eps(k)^2)
\]
and
\[
\frac{1}{\Gamma\left(\frac{1-\lambda(k)}{2}\right)}=\frac{1}{\Gamma\left(1-n-\frac{\eps(k)}{2}\right)}=-\frac{\eps(k)}{2}\left(\frac{1}{\Gamma}\right)'(1-n)+\mathcal{O}(\eps(k)^2).
\]
Substituting these Taylor expansions into \eqref{E:Devwronskien}, the condition $D_{\gb}(\lambda(k), k)=0$ implies that $\eps_{n}^{\pm}$ are the roots of the equation
\begin{equation}
\label{E:equationsurepsilon}
\hat{A}\eps^2+\eps (\hat{B}+\hat{C})+\hat{D}=0,
\end{equation}
where $(\hat{A},\hat{B},\hat{C},\hat{D})$ are analytic functions of $k$ respectively equivalent to the following functions as $k\to+\infty$:

\begin{equation}
\label{E:devABCD}
\left\{
\begin{aligned}
A(k)&=(-1)^{n+m}(n-1)!(m-1)!\pi2^{-\frac{1}{2}-\frac{(1+b)E_{n}}{4b}}b^{\frac{m}{2}-1}k^{-\frac{b+1}{2b}E_{n}},\\
B(k)&=(-1)^{n+m}(m-1)!\sqrt{\pi}2^{-3/2+\frac{(b-1)E_{n}}{4b}}b^{\frac{b+E_{n}}{4b}-1}(1+b)k^{\frac{b-1}{2b}E_{n}-2}e^{-k^2},\\
C(k)&=(-1)^{n+m}(n-1)!\sqrt{\pi}2^{-3/2+\frac{(1-b)E_{n}}{4b}}b^{\frac{b-E_{n}}{4b}-1}(1+b)k^{\frac{-b+1}{2b}E_{n}-2}e^{-k^2/b},\\
D(k)&=(-1)^{n+m-1}2^{\frac{E_{n}+b(2+E_{n})}{4b}}b^{\frac{b-E_{n}}{4b}}k^{\frac{b+1}{2b}E_{n}}e^{-k^2(1+1/b)}.
\end{aligned}
\right.
\end{equation}
The discriminant
\[
\Delta=(\hat{B}+\hat{C})^2-4\hat{A}\hat{D}
\]
is positive for $k$ large enough since there hold $\hat{B}(k)\gg\hat{C}(k)$ and $\hat{B}(k)^2 \gg \hat{D}(k)$ as $k$ goes to $+\infty$ (notice also that $\hat{A}(k)=o(1)$). In particular, the two roots of \eqref{E:equationsurepsilon} admit the following expansions:
\[
\eps^{+}= -\frac{\hat{D}}{\hat{B}}(1+o(1)) \ \ \mbox{and} \ \ \eps^{-}=-\frac{\hat{B}}{\hat{A}}(1+o(1)).
\]
Therefore, using the equivalents given in \eqref{E:devABCD}, we get as $k$ gets large:
\[
\eps^{+}(k)= \frac{b^{-m+\frac{3}{2}}2^{m+\frac{3}{2}}}{(m-1)!(1+b)\sqrt{\pi}}k^{2m+1}e^{-k^2/b}(1+o(1)),
\]
\[
\eps^{-}(k)=-2^{n-3/2}\frac{(1+b)}{(n-1)!\sqrt{\pi}}k^{2n-3}e^{-k^2}(1+o(1)).
\]
By definition, $\eps_{n}^{-}<\eps^{+}_{n}$. Due to the obvious sign of $\eps^{+}$ and $\eps^{-}$ for $k$ large enough, we get $\eps^{\pm}=\eps_{n}^{\pm}$ and the result follows from the above asymptotic expansions.

\appendix

\section{Parabolic cylinder functions}\label{appendix}
We list some properties of parabolic cylinder functions, and refer to~\cite[Chapter~19]{AbSt64}.
Given $a\in\R$, denote by $U(a,x)$ and $V(a,x)$ the so-called standard
solutions of the differential equation
\[
-y''(x)+\tfrac{1}{4}x^2 y(x)=-a y(x),\quad x\in\R.
\]
If follows by the series expansions of $U$ and $V$ that they are real-valued
if $a\in\R$ and $x\in\R$. Moreover, as $x\to+\infty$,
\begin{equation}
\label{eq:Uasymp}
U(a,x)\sim \exp(-x^2/4)x^{-a-1/2}
\sum_{j=0}^{+\infty} (-1)^j\frac{\bigl(\tfrac12+a\bigr)_{2j}}{j!(2x^2)^j}
\end{equation}
and
\begin{equation}
\label{eq:Vasymp}
V(a,x)\sim \sqrt{\tfrac{2}{\pi}}\exp(x^2/4)x^{a-1/2}
\sum_{j=0}^{+\infty}\frac{\bigl(\tfrac12-a\bigr)_{2j}}{j!(2x^2)^j}.
\end{equation}
Here $(x)_j=x(x-1)(x-2)\cdots(x-j+1)$ denotes the Pochhammer symbol.
It also holds that
\begin{equation}
\label{eq:UVrelation}
\pi V(a,x) = \Gamma(a+\tfrac12)\{\sin(\pi a) U(a,x)+U(a,-x)\}.
\end{equation}
This enables us to give an asymptotic expansion of $U(a,x)$ as $x\to-\infty$. Indeed, as $x\to+\infty$,
\begin{equation}
\label{eq:Uasympneg}
\begin{aligned}
U(a,-x) & \sim \frac{\sqrt{2\pi}}{\Gamma(a+\tfrac12)}\exp(x^2/4)x^{a-1/2}
\sum_{j=0}^{+\infty}\frac{\bigl(\tfrac12-a\bigr)_{2j}}{j!(2x^2)^j} \\
&\quad - \sin(\pi a)\exp(-x^2/4)x^{-a-1/2}
\sum_{j=0}^{+\infty} (-1)^j\frac{\bigl(\tfrac12+a\bigr)_{2j}}{j!(2x^2)^j}.
\end{aligned}
\end{equation}
If we keep only two terms in the sum,
\begin{equation}
\label{eq:Uasympneg2}
\begin{aligned}
U(a,-x) & \sim \frac{\sqrt{2\pi}}{\Gamma(a+\tfrac12)}\exp(x^2/4)x^{a-1/2}
\Bigl(
1+\tfrac{1}{2}(\tfrac12-a)_2\frac{1}{x^2}+\mathcal{O}(1/x^4)
\Bigr)\\
&\quad - \sin(\pi a)\exp(-x^2/4)x^{-a-1/2}
\Bigl(
1-\tfrac{1}{2}(a+\tfrac12)_2\frac{1}{x^2}+\mathcal{O}(1/x^4).
\Bigr).
\end{aligned}
\end{equation}
Next, it holds that
\[
U'(a,x):=\frac{d}{dx}U(a,x) = \frac{x}{2}U(a,x)-U(a-1,x),
\]
and thus
\[
U'(a,-x)=-\frac{x}{2}U(a,-x)-U(a-1,-x).
\]
Combining this with~\eqref{eq:Uasympneg}, we have, as $x\to+\infty$,
\begin{equation}
\label{eq:Uprimeasympneg}
\begin{aligned}
U'(a,-x) & \sim -\frac{\sqrt{2\pi}}{\Gamma(a+\tfrac12)}\exp(x^2/4)x^{a+1/2}
\times \\
&\qquad \biggl[\frac{1}{2}
\sum_{j=0}^{+\infty}\frac{\bigl(\tfrac12-a\bigr)_{2j}}{j!(2x^2)^j}
+\frac{a-\tfrac12}{x^2}\sum_{j=0}^{+\infty}\frac{(\tfrac{3}{2}-a)_{2j}}{j!(2x^2)^j}
\biggr]
\\
&\quad +\sin(\pi a)\exp(-x^2/4)x^{-a+1/2}
\times\\
&\qquad\biggl[
\frac{1}{2}\sum_{j=0}^{+\infty} (-1)^j\frac{\bigl(a+\tfrac12\bigr)_{2j}}{j!(2x^2)^j}
-\sum_{j=0}^{+\infty}(-1)^j\frac{\bigl(a-\tfrac{1}{2}\bigr)_{2j}}{j!(2x^2)^j}
\biggr].
\end{aligned}
\end{equation}
With two terms in the sum only,
\begin{equation}
\label{eq:Uprimeasympneg2}
\begin{aligned}
U'(a,-x) & \sim -\frac{\sqrt{2\pi}}{\Gamma(a+\tfrac12)}\exp(x^2/4)x^{a+1/2}
\times\\
&\qquad
\Bigl(
\frac{1}{2}+\Bigl\{\frac{1}{4}(\tfrac12-a)_2+(a-\tfrac12)\Bigr\}\frac{1}{x^2}+
\mathcal{O}(1/x^4)
\Bigr)\\
&\quad +\sin(\pi a)\exp(-x^2/4)x^{-a+1/2}
\times\\
&\qquad
\Bigl(
-\frac12 +\Bigl\{-\frac14(a+\tfrac12)_2+\frac12(a-\tfrac12)_2\Bigr\}\frac{1}{x^2}+\mathcal{O}(1/x^4)
\Bigr).
\end{aligned}
\end{equation}

\end{document}